\documentclass{amsart}

\usepackage{amsmath, amsfonts, amssymb, amsthm, amscd, graphicx, float, epstopdf, tabu
}
\usepackage[pdfpagemode={UseOutlines},bookmarks=true,bookmarksopen=true, bookmarksopenlevel=0,bookmarksnumbered=true,hypertexnames=false, colorlinks,linkcolor={blue},citecolor={blue},urlcolor={blue}, pdfstartview={FitV},unicode,breaklinks=true,backref=page]{hyperref}

\newtheorem{thm}{Theorem}[section]
\newtheorem{remark}[thm]{Remark}
\newtheorem{notation}[thm]{Notation}

\newtheorem{defn}[thm]{Definition}
\newtheorem{lem}[thm]{Lemma}
\newtheorem{prop}[thm]{Proposition}

\begin{document}

\title{Rank $n$ swapping algebra for Grassmannian}


\author{Zhe Sun}
\address{Department of mathematics, University of Luxembourg}
\email{sunzhe1985@gmail.com}
\thanks{The author was partially supported by the China Postdoctoral Science Foundation 2018T110084 and the Luxembourg National Research Fund(FNR) AFR bilateral grant COALAS 11802479-2. The author also acknowledge support from U.S. National Science Foundation grants DMS-1107452, 1107263, 1107367 “RNMS:GEometric structures And Representation varieties” (the GEAR Network).}


\date{}

\begin{abstract}
The {\em rank $n$ swapping algebra} is the Poisson algebra defined on the ordered pairs of points on a circle using the linking numbers, where a subspace of $(\mathbb{K}^n \times \mathbb{K}^{n*})^r/\operatorname{GL}(n,\mathbb{K})$ is its geometric model.
In this paper, we find an injective Poisson homomorphism from the Poisson algebra on Grassmannian $G_{n,r}$ arising from boundary measurement map to the rank $n$ swapping fraction algebra.
\end{abstract}

\maketitle


\section{Introduction}
We provide a geometric interpretation of bi-Poisson structure on the Grassmannian using a circle. 

In \cite{L18}, Labourie introduced the swapping algebra on ordered pairs of points on a circle using the linking numbers.
It was used to characterize the Goldman Poisson structure \cite{G84} on the character variety and the second Adler--Gel'fand--Dickey Poisson structure via Drinfel'd--Sokolov reduction \cite{DS81} on the space $Opers_n$ of $\operatorname{SL}(n,\mathbb{R})$-opers with trivial holonomy for any $n>1$. 
The ordered pairs of points on a circle should be understood as a pairing between a vector and a covector in a vector space $\mathbb{K}^n$ where $\mathbb{K}$ is a field of characteristic zero. When the dimension $n$ is fixed, by \cite{W39}\cite{CP76}, all these pairings generate the polynomial ring of a subspace of $(\mathbb{K}^n \times \mathbb{K}^{n*})^r/\operatorname{GL}(n,\mathbb{K})$. The relations among these pairings are generated by $(n+1)\times (n+1)$ determinant relations. In \cite{Su17}, we show that these polynomial relations define a Poisson ideal with respect to the swapping bracket and define the quotient ring $\mathcal{Z}_n(\mathcal{P})$ equipped with the swapping bracket, called {\em rank $n$ swapping algebra}. Actually, the swapping bracket depends on two parameters. We call it {\em $(\alpha,\beta)$-swapping bracket}, denoted by $\{\cdot,\cdot\}_{\alpha,\beta}$.

On the other hand, in \cite{GSV09}, Gekhtman et al. found a two dimensional family of Poisson brackets on the open Schubert cell of the Grassmannian induced from the boundary measurement map \cite{Po06}, denoted by $\{\cdot,\cdot\}_{B_{\alpha,\beta}}$. The parameters $(\alpha,\beta)$ are used to describe the $R$-matrix in \cite[Section 4]{GSV09}. Moreover, they show that these Poisson brackets are compatible with the natural cluster algebra structure \cite{FZ02}\cite{GSV03}, and the Grassmannian equipped with such Poisson bracket is a Poisson homogeneous space with respect to the natural action of $\operatorname{GL}_n$ equipped with an $R$-matrix Poisson--Lie structure \cite{Se83}. In \cite{GSSV12}, Gekhtman et al. show that every cluster algebra compatible Poisson structure on the open Schubert cell of the Grassmannian can be obtained as above.

Let $G_{n,r}^{\textbf{I}}$ be the open Schubert cell of the Grassmannian with respect to a $n$-element subset $\textbf{I}$ of $\{1,\cdots,r\}$ that is lexicographically minimal among all sets for which the Pl\"ucker coordinate is nonzero. Let $\mathcal{Q}_n(\mathcal{P})$ be the fraction field of $\mathcal{Z}_n(\mathcal{P})$ where $\#\mathcal{P}=r$. 
By sending the ratios of Pl\"ucker coordinates to the corresponding ratios of $(n\times n)$-determinants in $\mathcal{Q}_n(\mathcal{P})$
\begin{thm}[Main theorem]
There exists an injective Poisson algebra homomorphism from the coordinate ring of $G_{n,r}^{\textbf{I}}$ to $\mathcal{Q}_n(\mathcal{P})$ with respect to the Poisson bracket $\{\cdot,\cdot\}_{B_{\alpha,\beta}}$ and the $(\beta-\alpha,\alpha+\beta)$-swapping bracket.
\end{thm} 

In section \ref{section:rn}, we recall the swapping algebra and the rank $n$ swapping algebra. In section \ref{gnm}, we recall the Poisson structure on the Grassmannian induced from boundary measurement map and then prove the main theorem.
\section{Rank $n$ swapping algebra}
\label{section:rn}

In this section, we recall the swapping algebra \cite{L18} and the rank $n$ swapping algebra \cite{Su17}. We prove some basic facts using the linking numbers on a circle. Lemma \ref{swcal} (\cite[Lemma 3.5, Remark 3.6]{Su17}) is the key technical formula for the computation. 
\subsection{Swapping algebra}

Let $\mathcal{P}$ be a cyclic finite subset of $S^1$. We represent an ordered pair $(r, x)$ of $\mathcal{P}$ by the expression $rx$. Then we consider the associative commutative ring 
\begin{equation*}
\mathcal{Z}(\mathcal{P}) := \mathbb{K}[\{xy\}_{ x,y \in \mathcal{P}}]/\{xx\}_{x \in \mathcal{P}}
\end{equation*}
 over a field $\mathbb{K}$ of characteristic zero, where $\{xy\}_{x,y \in \mathcal{P}}$ are the set of variables. If $\#\mathcal{P}$ has $k$ points on $S^1$, then $\mathcal{Z}(\mathcal{P})$ is isomorphic to $\mathbb{K}[x_1,\cdots,x_{k^2-k}]$. Thus $\mathcal{Z}(\mathcal{P})$ is an integral domain.
Then we equip $\mathcal{Z}(\mathcal{P})$ with a swapping bracket depending on the linking numbers, defined as follows.

\begin{defn}
[Linking number]
\label{defnlkn}
Let $(r, x, s, y)$ be a quadruple of points in $\mathcal{P}\subset S^1$. Let $o$ be any point different from $r,x,s,y \in S^1$. 
Let $\sigma$ be a homeomorphism from $S^1\backslash o$ to $\mathbb{R}$ with respect to the \textbf{anticlockwise} orientation of $S^1$. Let $\triangle(a)= -1; 0; 1$ whenever $a < 0$; $a = 0$; $a > 0$ respectively.

The {\em linking number} between $rx$ and $sy$ is
\begin{equation*}
\begin{aligned}
&\mathcal{J}(rx, sy): = \frac{1}{2}  \cdot \triangle(\sigma(r)-\sigma(x)) \cdot \triangle(\sigma(r)-\sigma(y)) \cdot \triangle(\sigma(y)-\sigma(x)) 
\\&- \frac{1}{2} \cdot \triangle(\sigma(r)-\sigma(x)) \cdot \triangle(\sigma(r)-\sigma(s)) \cdot \triangle(\sigma(s)-\sigma(x)).
\end{aligned}
\end{equation*}

\end{defn}
\begin{figure}
\includegraphics[scale=0.5]{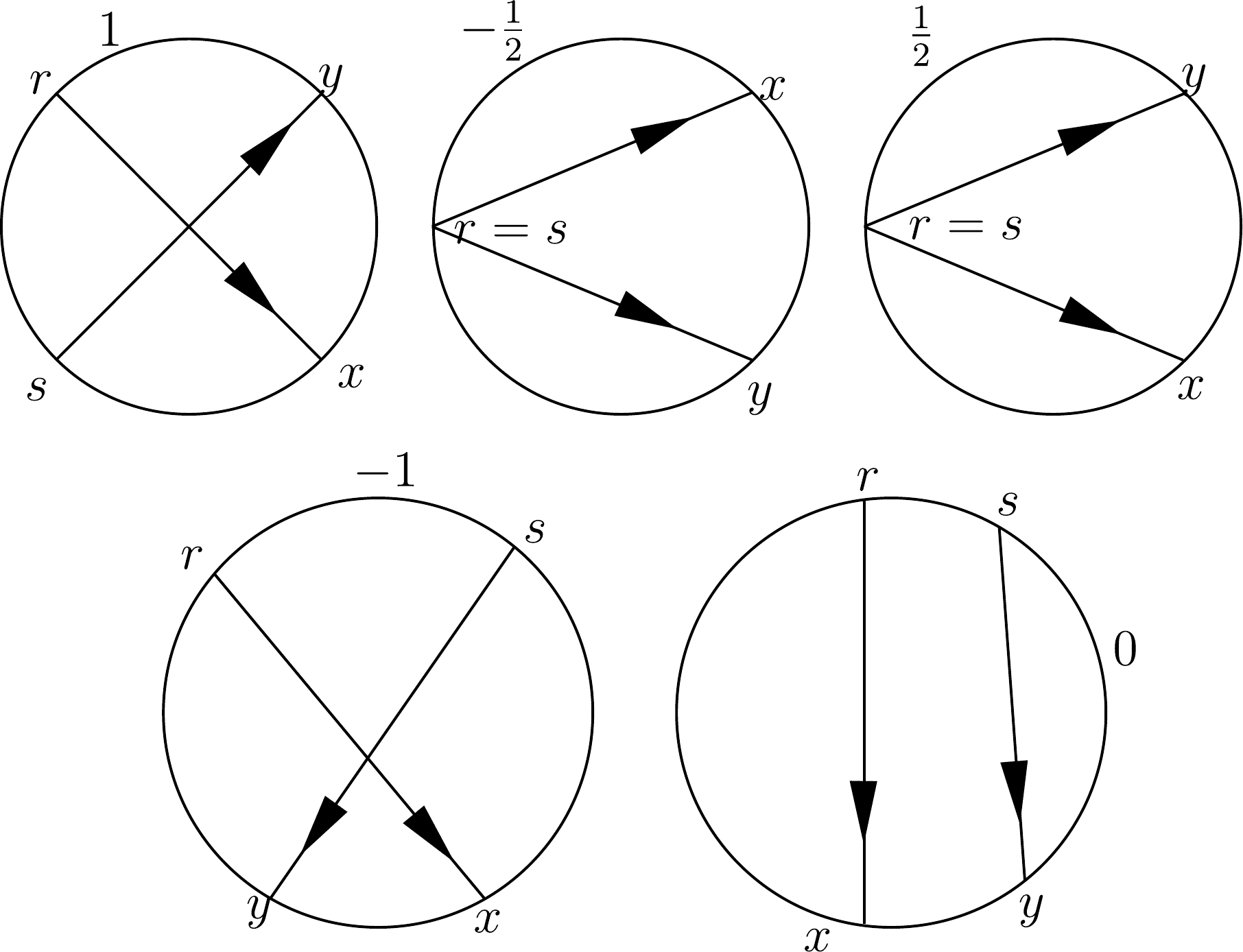}
\caption{Linking number $\mathcal{J}(rx, sy)$ between $rx$ and $sy$}
\label{swapswap1}
\end{figure}

In fact, the value of $\mathcal{J}(rx, sy)$ belongs to $\{0, \pm1, \pm \frac{1}{2}\}$, and does not depend on the choice of the point $o$ and depends only on the relative positions of $r,x,s,y$. In Figure \ref{swapswap1}, we describe five possible values of $\mathcal{J}(rx, sy)$.

\begin{defn}[$(\alpha,\beta)$-swapping bracket]
Let $\alpha,\beta \in \mathbb{K}$, the {\em $(\alpha,\beta)$-swapping bracket} over $\mathcal{Z}(\mathcal{P})$ is defined by extending the following formula on arbitrary generators $rx$, $sy$ to $\mathcal{Z}(\mathcal{P})$ using {\em Leibniz's rule}
\begin{equation*}
\{rx, sy\}_{\alpha,\beta} = \mathcal{J}(rx, sy) \cdot \left(\alpha \cdot ry \cdot sx + \beta \cdot rx \cdot sy\right).
\end{equation*}
We denote the $(1,0)$-swapping bracket by $\{\cdot,\cdot\}$, called {\em the swapping bracket}.
\end{defn}

It is easy to see that 
\begin{equation}
\label{equation:add}
\{\cdot, \cdot\}_{\alpha,\beta}= \alpha \{\cdot, \cdot\}_{1, 0}+ \beta \{\cdot, \cdot\}_{0, 1}.
\end{equation}
By direct computations, Labourie proved the following theorem.
\begin{thm}
\cite{L18}
\label{swappoisson}
The $(\alpha,\beta)$-swapping bracket is Poisson.
\end{thm}

Let $\mathcal{Q}(\mathcal{P})$ be fraction field of $\mathcal{Z}(\mathcal{P})$. The $(\alpha,\beta)$-swapping bracket can be extended to $\mathcal{Q}(\mathcal{P})$, since by Leibniz's rule, for any $A,B\in \mathcal{Z}(\mathcal{P})$
\begin{equation*}
\left\{A,\frac{1}{B}\right\}_{\alpha,\beta} = -\frac{\{A,B\}_{\alpha,\beta}}{B^2}.
\end{equation*}

\begin{defn}[Cross fraction]
Let $x,y,z,t$ belong to $\mathcal{P}$ so that $x\neq t$ and $y \neq z$. The {\em cross fraction} of $(x,y,z,t)$ is an element of $\mathcal{Q}(\mathcal{P})$: 
\begin{equation*}
\frac{xz}{xt} \cdot \frac{yt}{yz}.
\end{equation*}
Let $\mathcal{B}(\mathcal{P})$ be the subring of $\mathcal{Q}(\mathcal{P})$ generated by all the cross fractions.
\end{defn}

\begin{defn}
 {\em The $(\alpha,\beta)$-swapping fraction (multifraction resp.) algebra of $\mathcal{P}$} is $\mathcal{Q}(\mathcal{P})$ ($\mathcal{B}(\mathcal{P})$ resp.) equipped with the $(\alpha,\beta)$-swapping bracket, denoted by $(\mathcal{Q}(\mathcal{P}),\{\cdot, \cdot\}_{\alpha,\beta})$ ($(\mathcal{B}(\mathcal{P}),\{\cdot, \cdot\}_{\alpha,\beta})$ resp.).

\end{defn}
We will prove that $\mathcal{B}(\mathcal{P})$ is closed under $\{\cdot,\cdot\}_{\alpha,\beta}$ as follows.

\begin{lem}
\label{lem:swap01}
For any $ab, \frac{xz}{xt} \cdot \frac{yt}{yz} \in \mathcal{Q}(\mathcal{P})$, we have 
\begin{equation*}
\left\{ab, \frac{xz}{xt} \cdot \frac{yt}{yz}\right\}_{0,1}=0.
\end{equation*}

As a consequence, for any $\mu, \nu  \in \mathcal{B}(\mathcal{P})$, we obtain
\begin{equation*}
\{\mu, \nu\}_{0,1}=0.
\end{equation*}
\end{lem}
\begin{proof}
By Leibniz's rule, we have 
\begin{equation}
\label{equation:01}
\frac{\big\{ab, \frac{xz}{xt} \cdot \frac{yt}{yz}\big\}_{0,1}}{ab \cdot \frac{xz}{xt} \cdot \frac{yt}{yz}} 
=\frac{\big\{ab, xz\big\}_{0,1}}{ab \cdot xz}  - \frac{\big\{ab, xt\big\}_{0,1}}{ab \cdot xt} + \frac{\big\{ab, yt\big\}_{0,1}}{ab \cdot yt}  - \frac{\big\{ab, yz\big\}_{0,1}}{ab \cdot yz}.
\end{equation}
By the cocycle identity of the linking number \cite{L18} Definition 2.1.1(8):
\begin{equation*}
\mathcal{J}(ab,cd) + \mathcal{J}(ab,de) + \mathcal{J}(ab,ec)=0,
\end{equation*} 
the right hand side of Equation~\eqref{equation:01} equals
\begin{equation*}
\begin{aligned}
&\mathcal{J}(ab,xz) - \mathcal{J}(ab,xt) + \mathcal{J}(ab,yt) - \mathcal{J}(ab,yz)
\\&= \mathcal{J}(ab,tz) + \mathcal{J}(ab,zt) 
\\&=0.
\end{aligned}
\end{equation*}
\end{proof}
As a consequence, we obtain
\begin{prop}
The ring $\mathcal{B}(\mathcal{P})$ is closed under $\{\cdot,\cdot\}_{\alpha,\beta}$.
\end{prop}
\begin{proof}
For any $\mu, \nu \in \mathcal{B}(\mathcal{P})$, by Lemma \ref{lem:swap01}, we have
\begin{equation*}
\{\mu, \nu\}_{\alpha,\beta} 
=\alpha \cdot \{\mu, \nu\}_{1,0}+ \beta \cdot \{\mu, \nu\}_{0,1}
=\alpha \cdot \{\mu, \nu\}_{1,0}.
\end{equation*}
By \cite{Su17} Proposition 2.9, we obtain $\alpha \cdot \{\mu, \nu\}_{1,0}$ belongs to $\mathcal{B}(\mathcal{P})$. Hence we conclude that $\mathcal{B}(\mathcal{P})$ is closed under $\{\cdot,\cdot\}_{\alpha,\beta}$. \end{proof}
\subsection{Rank $n$ swapping algebra}
\label{rsa}
\begin{notation}
For any $d>1$ and any $x_1,\cdots,x_d, y_1, \cdots, y_d \in \mathcal{P}$, fix the notation
\begin{equation*}
\Delta\left((x_1,\cdots,x_d), (y_1, \cdots, y_d)\right): = \det   \left(\begin{array}{cccc}
       x_1 y_1 & \cdots & x_1 y_d \\
       \cdots & \cdots & \cdots \\
       x_d y_1 & \cdots & x_d y_d
     \end{array}\right).
\end{equation*} 
We call $(x_1,\cdots,x_d)$ ($(y_1,\cdots,y_d)$ resp) the {\em left (right resp.) side $n$-tuple of} the determinant $\Delta\left((x_1,\cdots,x_d), (y_1, \cdots, y_d)\right)$.
\end{notation}

\begin{defn}[Rank $n$ swapping ring $\mathcal{Z}_n(\mathcal{P})$]
For $n\geq 2$, let $R_n(\mathcal{P})$ be the ideal of $\mathcal{Z}(\mathcal{P})$ generated by 
\begin{align*}
\left\{ D \in \mathcal{Z}(\mathcal{P}) \; \bigg| 
\begin{array}{c}
\; D = \Delta\left((x_1,\cdots,x_{n+1}), (y_1, \cdots, y_{n+1})\right) ,\\ \forall x_1, \cdots, x_{n+1}, y_1,\cdots, y_{n+1} \in \mathcal{P}
\end{array}
 \right\}.
\end{align*}

The {\em rank n swapping ring} $\mathcal{Z}_n(\mathcal{P})$ is the quotient ring $\mathcal{Z}(\mathcal{P})/R_n(\mathcal{P})$.
\end{defn}

\begin{figure}
\includegraphics[scale=0.5]{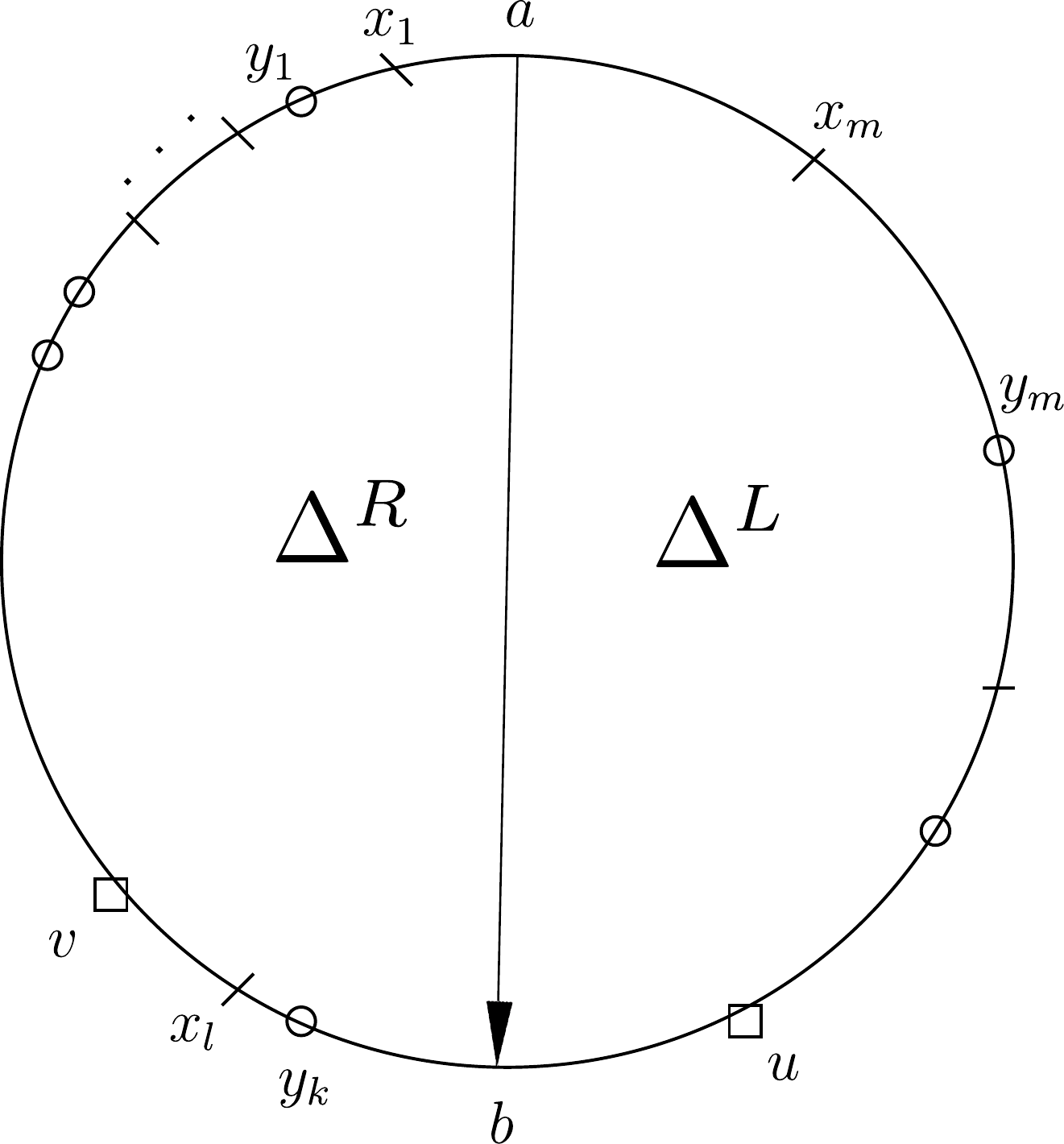}
\caption{$\{ab, \Delta((x_1, \cdots, x_{m}), (y_1,\cdots, y_{m}))\}$}
\label{swapabdet1}
\end{figure}

\begin{lem}\cite[Lemma 3.5, Remark 3.6]{Su17}
\label{swcal}
For any integer $m \geq 2$, suppose $x_1,\cdots,x_{m}$ $(y_1,\cdots,y_{m} \; resp. )$ in $\mathcal{P}$ are mutually distinct and anticlockwise ordered. Assume that $a,b$ belong to $\mathcal{P}$ and $x_{1},\cdots,x_{l}, y_{1},\cdots$, $y_{k}$ are on the \textbf{right} side of the oriented edge $\overrightarrow{ab}$ (include coinciding with $a$ or $b$) as illustrated in Figure \ref{swapabdet1}.
  Let $u$ ($v$ resp.) be strictly on the left (right resp.) side of $\overrightarrow{ab}$. Let
\begin{equation}
\begin{aligned}
\label{equation:R}
&\Delta^R(ab)= \sum_{d=1}^l \mathcal{J}(ab, x_d u)\cdot x_d b \cdot \Delta((x_1, \cdots,x_{d-1}, a, x_{d+1},\cdots,x_{m}), (y_1,\cdots, y_{m}))
  \\&+ \sum_{d=1}^k \mathcal{J}(ab, u y_d)\cdot a y_d \cdot \Delta((x_1, \cdots,x_{m}), (y_1,\cdots,y_{d-1}, b, y_{d+1},\cdots, y_{m})),
\end{aligned}
\end{equation}
\begin{equation*}
\begin{aligned}
&\Delta^L(ab)= \sum_{d=k+1}^{m} \mathcal{J}(ab, x_d v)\cdot x_d b \cdot \Delta((x_1, \cdots,x_{d-1}, a, x_{d+1},\cdots,x_{m}), (y_1,\cdots, y_{m}))
  \\&+ \sum_{d=l+1}^{m} \mathcal{J}(ab, v y_d)\cdot a y_d \cdot \Delta((x_1, \cdots,x_{m}), (y_1,\cdots,y_{d-1}, b, y_{d+1},\cdots, y_{m})),
\end{aligned}
\end{equation*}

then we have
\begin{equation*}
\{ab, \Delta((x_1, \cdots, x_{m}), (y_1,\cdots, y_{m}))\}
   = \Delta^R(ab)=\Delta^L(ab).
\end{equation*}

\end{lem}
\begin{lem}
\label{swc2}
With the same notations as in the lemma above, we have
\begin{eqnarray*}
\{ab, \Delta((x_1, \cdots, x_{m}), (y_1,\cdots, y_{m}))\}_{0,1}
= K\cdot  ab \cdot \Delta((x_1, \cdots, x_{m}), (y_1,\cdots, y_{m})),
\end{eqnarray*}
for some constant $K$.
\end{lem}
\begin{proof}
Since 
\begin{equation*}
\Delta((x_1, \cdots, x_{m}), (y_1,\cdots, y_{m})) = \sum_{\sigma \in S_{m}} \epsilon_\sigma \prod_i x_i y_{\sigma(i)}
\end{equation*}
 where $\epsilon_\sigma$ is the sign of $\sigma$ in the permutation group $S_{m}$, 
\begin{equation*}
\begin{aligned}
&\{ab, \Delta((x_1, \cdots, x_{m}), (y_1,\cdots, y_{m}))\}_{0,1} 
\\&= ab \cdot \sum_{\sigma \in S_{m}} \epsilon_\sigma \big(\sum_{i=1}^{m}\mathcal{J}(ab, x_i y_{\sigma(i)})\big) \prod_i x_i y_{\sigma(i)}
\\&= ab \cdot \sum_{\sigma \in S_{m}} \epsilon_\sigma K(\sigma) \prod_i x_i y_{\sigma(i)},
\end{aligned}
\end{equation*} 
where $K(\sigma)=\sum_{i=1}^{m}\mathcal{J}(ab, x_i y_{\sigma(i)})$. For any transposition $(kl)$ in the permutation group $S_{m}$, we get 
\begin{equation*}
\begin{aligned}
&K((kl)\cdot \sigma) - K(\sigma) 
\\&= \mathcal{J}(ab, x_k y_{\sigma(l)})+ \mathcal{J}(ab, x_l y_{\sigma(k)}) - \mathcal{J}(ab, x_k y_{\sigma(k)}) - \mathcal{J}(ab, x_l y_{\sigma(l)})
\\&= \mathcal{J}(ab, y_{\sigma(k)} y_{\sigma(l)})+  \mathcal{J}(ab, y_{\sigma(l)} y_{\sigma(k)})
\\&= 0.
\end{aligned}
\end{equation*} 
For any $\sigma', \sigma \in S_{m}$, $\sigma'$ is related to $\sigma$ by a sequence of transpositions.
Thus $K(\sigma)=K(\sigma')$. Then we denote $K(\sigma)$ by $K$.
Finally, we conclude that 
\begin{equation*}
\{ab, \Delta((x_1, \cdots, x_{m}), (y_1,\cdots, y_{m}))\}_{0,1}
= K\cdot  ab \cdot \Delta((x_1, \cdots, x_{m}), (y_1,\cdots, y_{m})).
\end{equation*}

\end{proof}
The following proposition is a consequence of Lemma \ref{swcal} and Lemma \ref{swc2}.
\begin{prop}
$R_n(\mathcal{P})$ is a Poisson ideal of $\mathcal{Z}(\mathcal{P})$ with respect to the $(\alpha,\beta)$-swapping bracket.
\end{prop}

\begin{defn}[$(\alpha,\beta)$-rank $n$ swapping algebra of $\mathcal{P}$]
The {\em $(\alpha,\beta)$-rank $n$ swapping algebra of $\mathcal{P}$} is the ring $\mathcal{Z}_n(\mathcal{P})=\mathcal{Z}(\mathcal{P})/R_n(\mathcal{P})$ equipped with the $(\alpha,\beta)$-swapping bracket, denoted by $(\mathcal{Z}_n(\mathcal{P}),\{\cdot, \cdot\}_{\alpha,\beta})$.

We define $(\mathcal{Z}_n(\mathcal{P}),\{\cdot, \cdot\})$ to be the {\em rank $n$ swapping algebra of $\mathcal{P}$}.  
\end{defn}

By Theorem 4.7 in \cite{Su17}, $\mathcal{Z}_n(\mathcal{P})$ is an integral domain. Thus the fraction field of $\mathcal{Z}_n(\mathcal{P})$ is well-defined, denoted by $\mathcal{Q}_n(\mathcal{P})$. The cross fractions are well-defined in $\mathcal{Q}_n(\mathcal{P})$.
Let $\mathcal{B}_n(\mathcal{P})$ be the subring of $\mathcal{Q}_n(\mathcal{P})$ generated by all the cross fractions.

\begin{defn}
\label{defnsma}
The {\em $(\alpha,\beta)$-rank $n$ swapping fraction (multifraction resp.) algebra of $\mathcal{P}$} is $\mathcal{Q}_n(\mathcal{P})$ ($\mathcal{B}_n(\mathcal{P})$ resp.) equipped with the $(\alpha,\beta)$-swapping bracket, denoted by $(\mathcal{Q}_n(\mathcal{P}),\{\cdot, \cdot\}_{\alpha,\beta})$ ($(\mathcal{B}_n(\mathcal{P}),\{\cdot, \cdot\}_{\alpha,\beta})$ resp.).

\end{defn}

\subsection{$(n\times n)$-determinant ratio}
Let us recall the $(n\times n)$-determinant ratio in \cite[Section 4]{Su15}.

\begin{defn}[$(n\times n)$-determinant ratio]
\label{defndr}
Let $x_1,\cdots,x_{n-1},y \in \mathcal{P}$ be mutually distinct.
The {\em $(n\times n)$-determinant ratio} of $x_1,\cdots,x_{n-1},t,y$:
\begin{equation*}
E(x_1,\cdots,x_{n-1}| t,y) := \frac{\Delta\left(\left(x_1,\cdots, x_{n-1} , t\right), \left(v_1,\cdots, v_n\right)\right)}{\Delta\left(\left(x_1, \cdots,x_{n-1} , y\right), \left(v_1,\cdots, v_n\right)\right)}
\end{equation*}
for any mutually distinct $v_1,\cdots,v_n \in \mathcal{P}$. By~\cite[Corollary 4.5]{Su15}, the $(n\times n)$-determinant ratio does not depend on the mutually distinct $v_1,\cdots,v_n \in \mathcal{P}$ that we choose.

Let $\mathcal{DR}_n(\mathcal{P})$ be the subfield of $\mathcal{Q}_n(\mathcal{P})$ generated by all the $(n \times n)$-determinant ratios. 
\end{defn}

\begin{remark}
\label{remark:rnt}
As a consequence, the swapping bracket between any element in $\mathcal{Q}_n(\mathcal{P})$ and $E(x_1,\cdots,x_{n-1}| t,y)$ does not depend on the choice of right side $n$-tuple $(v_1,\cdots,v_n)$ of $E(x_1,\cdots,x_{n-1}| t,y)$.
We can calculate the swapping bracket between two $(n\times n)$-determinant ratios with the right side $n$-tuples in any preferred positions by enlarging $\mathcal{P}$.
\end{remark}

\section{Rank $n$ swapping algebra on the Grassmannian}
\label{gnm}
\subsection{Poisson structure on the Grassmannian}
\begin{defn}[Grassmannian]
Let $n$ and $r$ be two integers such that $0<n\leq r$. Let $\mathbb{K}$ be a field of characteristic zero (e.g. $\mathbb{R}$ or $\mathbb{C}$). The {\em Grassmannian} $\textbf{Gr}(n,r)$ is the manifold of $n$-dimensional vector subspaces in $\mathbb{K}^r$.
\end{defn}

Let $\operatorname{Mat}^*(n,r)$ be the set of $(n\times r)$-matrices over $\mathbb{K}$ with rank $n$. Then
\begin{equation*}
\textbf{Gr}(n,r)=\operatorname{GL}_n\backslash \operatorname{Mat}^*(n,r)
\end{equation*}
where $\operatorname{GL}_n$ acts on $\operatorname{Mat}^*(n,r)$ by left multiplication.

\begin{defn}[Coordinates on $\textbf{Gr}(n,r)$]
Let $\mathbf{I}$ be a $n$-element subset of $\{1,\cdots,r\}$. The {\em Pl\"ucker coordinate} $\Delta_\mathbf{I}$ of $\operatorname{Mat}^*(n,r)$ is the minor of the matrix formed by the columns of the matrix indexed by $\mathbf{I}$ with respect to lexicographical order. 

The {\em Schubert cell} 
\begin{equation*}
\textbf{Gr}(n,r)^{\mathbf{I}}:=\{A \in \textbf{Gr}(n,r) \;| \; \mathbf{I}\;\;\text{is lexicographically minimal such that } \Delta_\mathbf{I}(A)\neq 0\}.
\end{equation*}
The Grassmannian $\textbf{Gr}(n,r)$ is a disjoint union $\bigcup_{\mathbf{I}} \textbf{Gr}(n,r)^{\mathbf{I}}$.

Observe that for any $g \in \operatorname{GL}_n$,
\begin{equation*}
\Delta_\mathbf{I}(g\cdot A)= \det(g)\cdot \Delta_\mathbf{I}(A).
\end{equation*}
The subset $\mathbf{I}(i \rightarrow j)$ is obtained from $\mathbf{I}$ by replacing $i\in \mathbf{I}$ by $j\notin \mathbf{I}$. Then 
\begin{equation*}
\left\{m_{ij} = \frac{\Delta_{\mathbf{I}(i \rightarrow j)}}{\Delta_\mathbf{I}}\right\}_{i\in \mathbf{I}, j \notin \mathbf{I}}
\end{equation*}
form a {\em coordinate system} on the Schubert cell $\textbf{Gr}(n,r)^{\mathbf{I}}$.
\end{defn}

\begin{figure}
\includegraphics[scale=0.5]{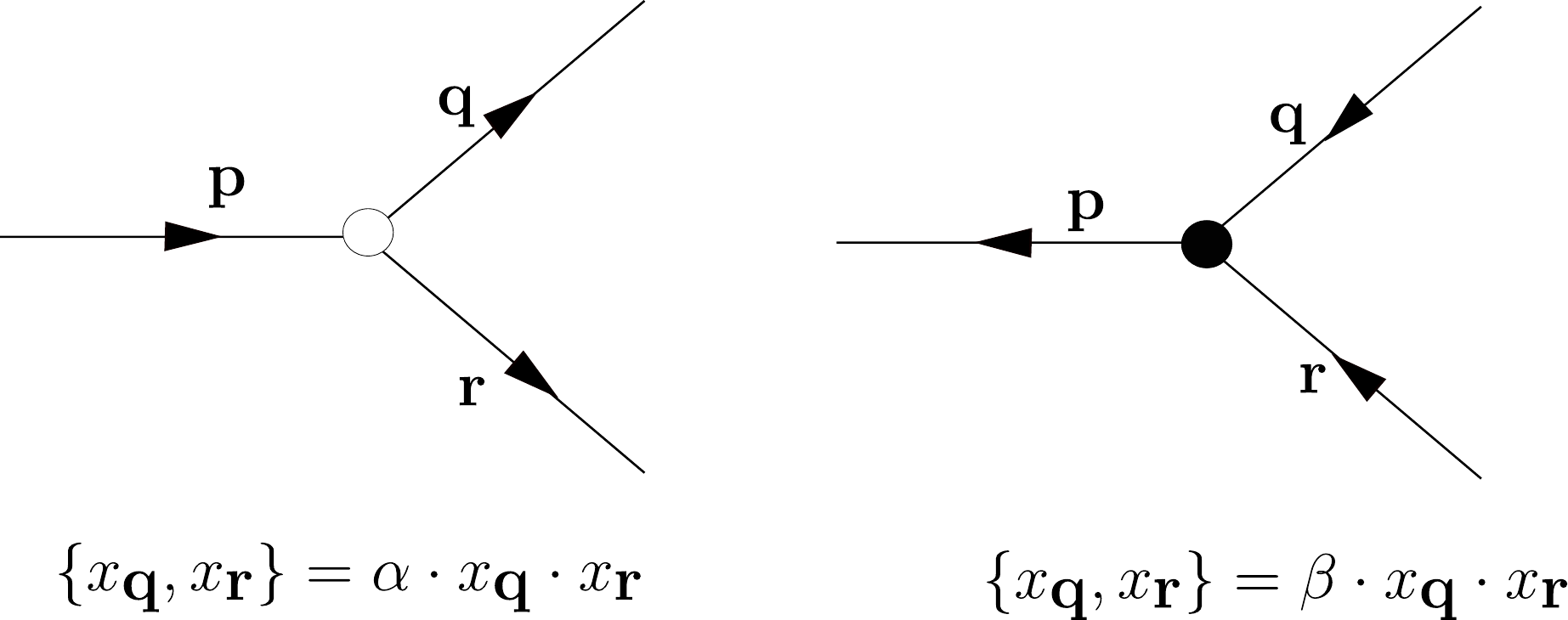}
\caption{}
\label{figure:bw}
\end{figure}

\begin{defn}[Directed planar graph for $\textbf{Gr}(n,r)$]
A {\em directed planar graph for $\textbf{Gr}(n,r)$} is a directed planar graph $(V,E)$ on a disk with $r$ vertices $v_1,\cdots,v_r$ ordered anticlockwise on the boundary, where there are $n$ ($r-n$ resp.) vertices on the boundary such that each one of these vertices has exactly one outgoing (incoming resp.) edge, each inner vertex has one incoming edge and two outgoing edges (called white vertex) or two incoming edges and one outgoing edge (called black vertex) as in Figure~\ref{figure:bw}. 
\end{defn}
To a vertex $v$ and the edge $e$ connecting to $v$, let $\textbf{p}$ be $(v,e)$, we assign a variable $x_\textbf{p} \in \mathbb{R}\backslash \{0\}$. Let $d$ be the cardinality of these pairs. We define a Poisson structure on $(\mathbb{R}\backslash \{0\})^d$ as follows. For a white (black resp.) vertex, labelling from incoming (outgoing resp.) edge clockwise $x_\textbf{p}, x_\textbf{q}, x_\textbf{r}$, as in Figure~\ref{figure:bw}, the Poisson bracket is $\{x_\textbf{q}, x_\textbf{r}\}_N = \alpha \cdot x_\textbf{q} \cdot x_\textbf{r}$ ($\{x_\textbf{q}, x_\textbf{r}\}_N = \beta \cdot x_\textbf{q} \cdot x_\textbf{r}$ resp.), otherwise zero.

\begin{defn}[Perfect planar network]
Given a directed planar graph $(V,E)$ for $\textbf{Gr}(n,r)$, let $w$ be a map
\begin{equation*}
w : (\mathbb{R}\backslash \{0\})^d \rightarrow\mathbb{R}^{|E|}
\end{equation*} 
such that the entry for the edge $e$ is $w_e = x_\textbf{a} x_\textbf{b}$ where $\textbf{a} = (v, e)$ and $\textbf{b} = (u, e)$. We define the {\em perfect planar network} to be $N = (V,E, \{w_e\}_{e \in E})$. Each perfect planar network defines the {\em space of edge weights} $E_N:=w((\mathbb{R}\backslash \{0\})^d) $. 
\end{defn}

\begin{defn}[Boundary measurement map]
Let $N = (V,E, \{w_e\}_{e \in E})$ be the perfect planar network where the $n$-element subset $\mathbf{I}$ corresponds to the vertices on the boundary which have exactly one outgoing edge.
The boundary measurement map $b$ is a rational map from $E_N$ to the cell $\textbf{Gr}(n,r)^\mathbf{I}$ given as follows. Choose a unique representative in $\operatorname{Mat}^*(n,r)$ so that the sub $n\times n$ matrix formed by the columns from $\mathbf{I}$ is the identity matrix $I_{n\times n}$, the other $(i,j)$-entry of the $(n\times r)$-matrix is defined to be the sum of the products of the weights of all paths starting at $v_i$ and ending at $v_j$. 
\end{defn}
The map $b \circ w$ induces the $\{\cdot, \cdot\}_{B_{\alpha,\beta}}$ Poisson bracket on $\textbf{Gr}(n,r)^\mathbf{I}$ from the Poisson bracket $\{\cdot, \cdot\}_N$ on $(\mathbb{R}\backslash \{0\})^d$.

We identify the set $\{1,2,\cdots,r\}$ used to enumerate the columns with the anticlockwise ordered set $\mathcal{P}=\{a_1,a_2,\cdots,a_r\}$ on a circle.
\begin{defn}
\cite[Theorem 3.3, $\{\cdot, \cdot\}_{B_{\alpha,\beta}}$]{GSV09}
Let $\mathcal{P} = \{a_1, a_2, \cdots a_r\}$ be ordered anticlockwise on a circle. The symbol $\prec$ is used to denote the anticlockwise ordering on a circle. Let {\em parallel number} be
\begin{equation*}
s_{||}(a_i a_j,a_{i'} a_{j'})=\left\{
\begin{aligned}
&1 \;\;\; if \;\;\; a_i\prec a_{i'}\prec a_{j'} \prec a_j \prec a_i;  \\
&-1 \;\;\; if \;\;\; a_{i'}\prec a_i\prec a_j \prec a_{j'} \prec a_{i'};   \\
&\frac{1}{2} \;\;\; if \;\;\; a_i= a_{i'}\prec a_{j'} \prec a_j \prec a_i \; or \; a_i\prec a_{i'}\prec a_{j'} = a_j \prec a_i;   \\
&-\frac{1}{2} \;\;\; if \;\;\; a_{i'}= a_i\prec a_j \prec a_{j'} \prec a_{i'} \; or \; a_{i'}\prec a_i\prec a_j = a_{j'}\prec  a_{i'}; \\
&0 \;\;\; otherwise.
\end{aligned}
\right.
\end{equation*}
As in Figure \ref{figure:parallel}, the parallel number depends only on the corresponding position of the four points.
For any $n$-element subset $\mathbf{I}$, the $\{\cdot, \cdot\}_{B_{\alpha,\beta}}$ Poisson bracket of $\textbf{Gr}(n,r)^{\mathbf{I}}$ is given by
\begin{equation}
\label{equation:pgnr}
\{m_{ij}, m_{i'j'}\}_{B_{\alpha,\beta}} = (\alpha-\beta)\cdot s_{||}(a_i a_j, a_{i'} a_{j'}) \cdot m_{ij'} \cdot m_{i'j}+ (\alpha+\beta)\cdot \mathcal{J}(a_i a_j, a_{i'} a_{j'}) \cdot m_{ij} \cdot m_{i'j'}.
\end{equation}
\end{defn}

\begin{remark}
When $\textbf{I}=\{1,2,\cdots,n\}$, the above Poisson bracket can be also found in \cite{BGY06}. The parameters $(\alpha,\beta)$ are used to describe the $R$-matrix in \cite[Section 4]{GSV09}. The definition of $s_{\times}$ from \cite{GSV09} coincides with the linking number $\mathcal{J}$.
\end{remark}

\begin{figure}
\includegraphics[scale=0.5]{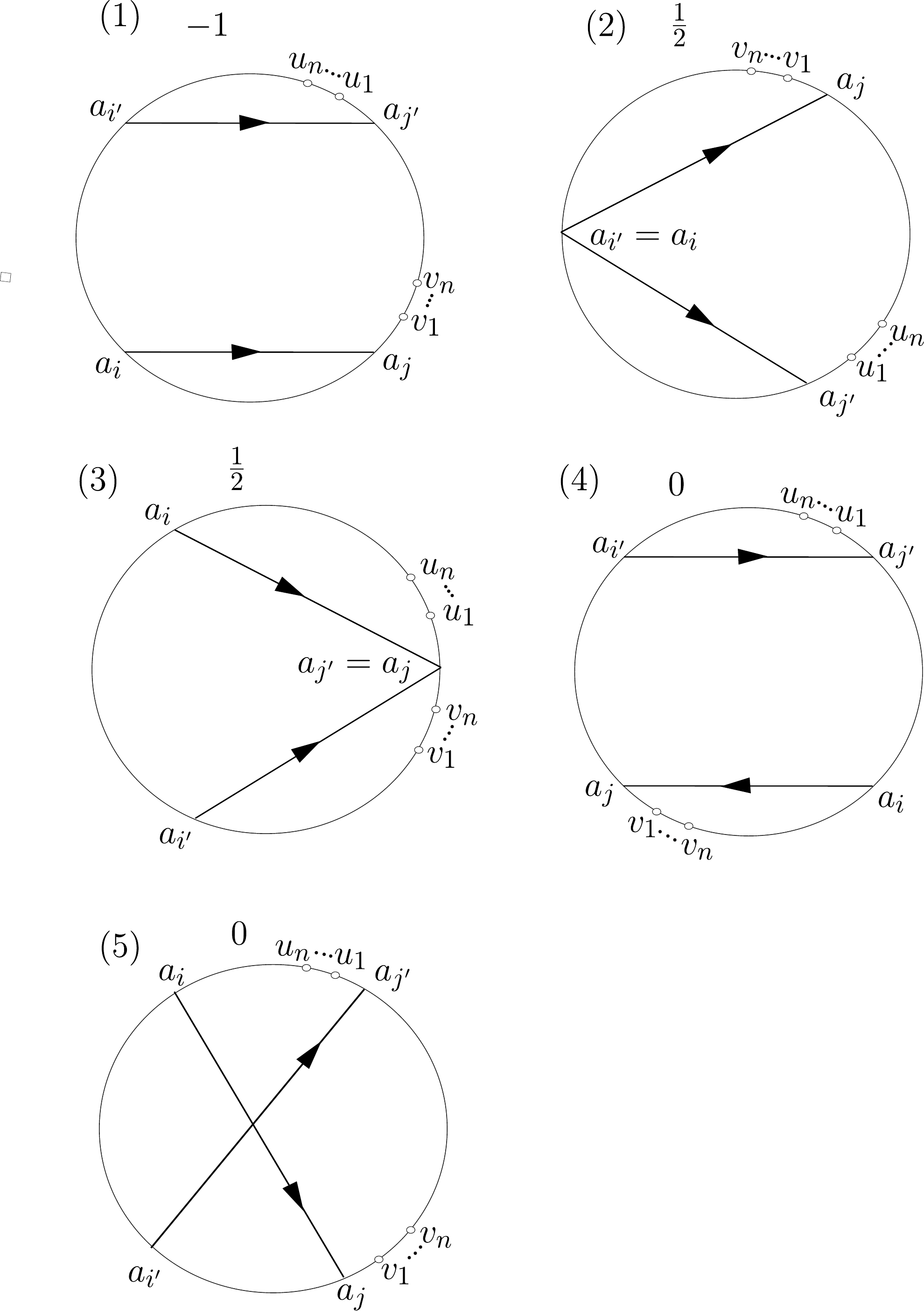}
\caption{Parallel number $s_{||}(a_i a_j, a_{i'} a_{j'})$, arrangement of $(v_1,\cdots,v_n)$ and $(u_1,\cdots,u_n)$}
\label{figure:parallel}
\end{figure}

\subsection{Main theorem}
\begin{thm}[Main theorem]
\label{theorem:Gab}
Let $\mathcal{P} = \{a_1\prec \cdots \prec a_r \prec a_1\}$ be ordered anticlockwise on a circle. The $(\beta-\alpha, \alpha+\beta)$-rank $n$ swapping multifraction algebra is denoted by $(\mathcal{Q}_n(\mathcal{P}),\{\cdot, \cdot\}_{\beta-\alpha, \alpha+\beta})$. Let $\mathbf{I}$ be a $n$-element subset $\mathbf{I}= \{k_1, \cdots, k_n\}$ of $\{1,\cdots, r\}$ with $k_1<\cdots<k_n$. If $i \in \mathbf{I}$, let $\widehat{a_i}$ be $a_{k_1}, \cdots, a_{k_n}$ with $a_i$ removed.
 There is an injective algebra homomorphism 
\begin{equation*} 
\theta_{\alpha,\beta}: \mathbb{K}[\{m_{ij}\}_{i\in \mathbf{I}, j\notin \mathbf{I}}] \rightarrow \mathcal{Q}_n(\mathcal{P})
\end{equation*}
defined by extending the following formula on arbitrary generator $m_{ij}$ to the polynomial ring
\begin{equation*}
\theta_{\alpha,\beta}\left(m_{ij}\right)= E\left(\widehat{a_i}|a_j,a_i\right).
\end{equation*}
Then the algebra homomorphism $\theta_{\alpha,\beta}$ is Poisson with respect to $\{\cdot,\cdot\}_{B_{\alpha,\beta}}$ and $\{\cdot, \cdot\}_{\beta-\alpha,\alpha+\beta}$.
\end{thm}
We prove the main theorem by several steps. Firstly, we send the ratios of $(n\times n)$-determinants in $\mathbb{K}^n$ to the $(n\times n)$-determinant ratios for $\mathcal{Z}_n(\mathcal{P})$. The injectivity follows the same argument in \cite[Proposition 5.3]{Su15} using \cite[Theorem 4.6]{Su17}.
\begin{notation}
Let us denote $\frac{\{A,B\}_{\alpha,\beta}}{A\cdot B}$ by $[A,B]_{\alpha,\beta}$. 

Moreover, we denote $[A,B]_{1,0}$ by $[A,B]$.

Let $u^k:=u_1,\cdots,u_k$.
\end{notation}

\begin{lem}
\label{lemma:01}
\begin{equation*}
\begin{aligned}
&\left\{E\left(\widehat{a_i}|a_j,a_i\right), E\left(\widehat{a_{i'}}|a_{j'},a_{i'}\right)\right\}_{0,1}
=  \mathcal{J}(a_i a_j, a_{i'} a_{j'})  \cdot  E\left(\widehat{a_i}|a_j,a_i\right) \cdot E\left(\widehat{a_{i'}}|a_{j'},a_{i'}\right),
\end{aligned}
\end{equation*}
\end{lem}
\begin{proof}
Let $(u_1,\cdots,u_n)$ be the right side $n$-tuple of the two $(n\times n)$-determinant ratios.
By Leibniz's rule, we obtain
\begin{equation}
\label{equation:leibEE}
\begin{aligned}
&\left[E\left(\widehat{a_i}|a_j,a_i\right), E\left(\widehat{a_{i'}}|a_{j'},a_{i'}\right)\right]_{0,1}
\\&=\left[\Delta\left((\widehat{a_i}, a_j),(u^n)\right), \Delta\left((\widehat{a_{i'}},a_{j'}),(u^n)\right)\right]-\left[\Delta\left((\widehat{a_i}, a_j),(u^n)\right), \Delta\left((\widehat{a_{i'}},a_{i'}),(u^n)\right)\right]
\\&-\left[\Delta\left((\widehat{a_i}, a_i),(u^n)\right), \Delta\left((\widehat{a_{i'}},a_{j'}),(u^n)\right)\right]+\left[\Delta\left((\widehat{a_i}, a_i),(u^n)\right), \Delta\left((\widehat{a_{i'}},a_{i'}),(u^n)\right)\right]
\end{aligned}
\end{equation}

By Lemma \ref{swc2}, we have
\begin{equation*}
\begin{aligned}
&\left\{E\left(\widehat{a_i}|a_j,a_i\right), E\left(\widehat{a_{i'}}|a_{j'},a_{i'}\right)\right\}_{0,1}
= K  \cdot  E\left(\widehat{a_i}|a_j,a_i\right) \cdot E\left(\widehat{a_{i'}}|a_{j'},a_{i'}\right),
\end{aligned}
\end{equation*}
where $K$ is some constant. By Equation~\eqref{equation:leibEE}, using the same argument as in Lemma \ref{swc2}, $K$ is the summation of $4\cdot n^2$ linking numbers. We observe that all the terms associated to $\widehat{a_{i'}}$ or $\widehat{a_{i}}$ are cancelled except the following four terms:
\begin{equation*}
\begin{aligned}
&K=   \mathcal{J}(a_j u_n, a_{j'} u_n) - \mathcal{J}(a_j u_n, a_{i'} u_n) - \mathcal{J}(a_i u_n, a_{j'} u_n) + \mathcal{J}(a_i u_n, a_{i'} u_n)
\\&=  \mathcal{J}(a_i a_j, a_{i'} a_{j'}).
\end{aligned}
\end{equation*}
\end{proof}

By Equation~\eqref{equation:add}, to prove Theorem \ref{theorem:Gab}, it is enough to prove 
\begin{equation*}
\begin{aligned}
&\left\{E\left(\widehat{a_i}|a_j,a_i\right), E\left(\widehat{a_{i'}}|a_{j'},a_{i'}\right)\right\}_{1,0}
= - s_{||}\left(a_i a_j,a_{i'} a_{j'}\right)\cdot E\left(\widehat{a_i}|a_{j'},a_i\right) \cdot E\left(\widehat{a_{i'}}|a_{j},a_{i'}\right).
\end{aligned}
\end{equation*}

The main technique that we use to prove the theorem is the following lemma.
\begin{lem}\cite[Lemma 5.6]{Su15}
\label{lemcalMB}
For $n\geq 2$, let $M=\left(c_s d_t\right)_{s,t=1}^n$ be a $(n\times n)$-matrix with $c_s, d_t \in \mathcal{P}$, let $M_{st}$ be the determinant of the matrix obtained from $M$ by deleting the $s$-th row and the $t$-th  column. Let $B \in \mathcal{Q}_n(\mathcal{P})$, we have
\begin{equation*}
\left\{\det M, B\right\} = \sum_{s=1}^n \sum_{t=1}^n (-1)^{s+t} \cdot \det M_{st} \cdot \left\{c_s d_t, B\right\}
\end{equation*}
in $\mathcal{Q}_n(\mathcal{P})$.
\end{lem}

\begin{proof}[Proof of Theorem \ref{theorem:Gab}]
We prove the case (1) in Figure \ref{figure:parallel}. 

 Suppose that $(v_1,\cdots,v_n)$ ($(u_1,\cdots,u_n)$ resp.) is the right side $n$-tuple for the term on the left (right resp.) hand side in each $[\cdot,\cdot]$ bracket below. 
By Remark \ref{remark:rnt}, we arrange the points $(v_1,\cdots,v_n)$ ($(u_1,\cdots,u_n)$ resp.) in between two successive points of $\mathcal{P}$ where one of them is $a_j$ ($a_{j'}$ resp.) as in Figure \ref{figure:parallel}. 
For all the cases, such arrangements simplify the computation a lot, and allows us to compute in a similar way.

We compute $D=\left[\Delta\left(\left(\widehat{a_i},a_j\right), \left(v^n\right)\right), \Delta\left(\left(\widehat{a_{i'}},a_{j'}\right), \left(u^n\right)\right)\right]$. By Lemma \ref{lemcalMB}, for $\det M = \Delta\left(\left(\widehat{a_i},a_j\right), \left(v^n\right)\right)$ and $B =  \Delta\left(\left(\widehat{a_{i'}},a_{j'}\right), \left(u^n\right)\right)$, we have 
\begin{equation*}
D=\frac{1}{\det M \cdot B}\sum_{s=1}^n \sum_{t=1}^n (-1)^{s+t} M_{st} \left\{c_s v_t, B\right\},
\end{equation*}
 where $c_s\in\{\widehat{a_i},a_j\}$. We fix $c_s$ and compute the sum 
 \[\frac{1}{\det M \cdot B}\sum_{t=1}^n (-1)^{s+t} M_{st} \left\{c_s v_t, B\right\}\]
over $t$, where the summation is called {\em the sum over $t$ for $c_s$} for short:

\begin{figure}
\includegraphics[scale=0.7]{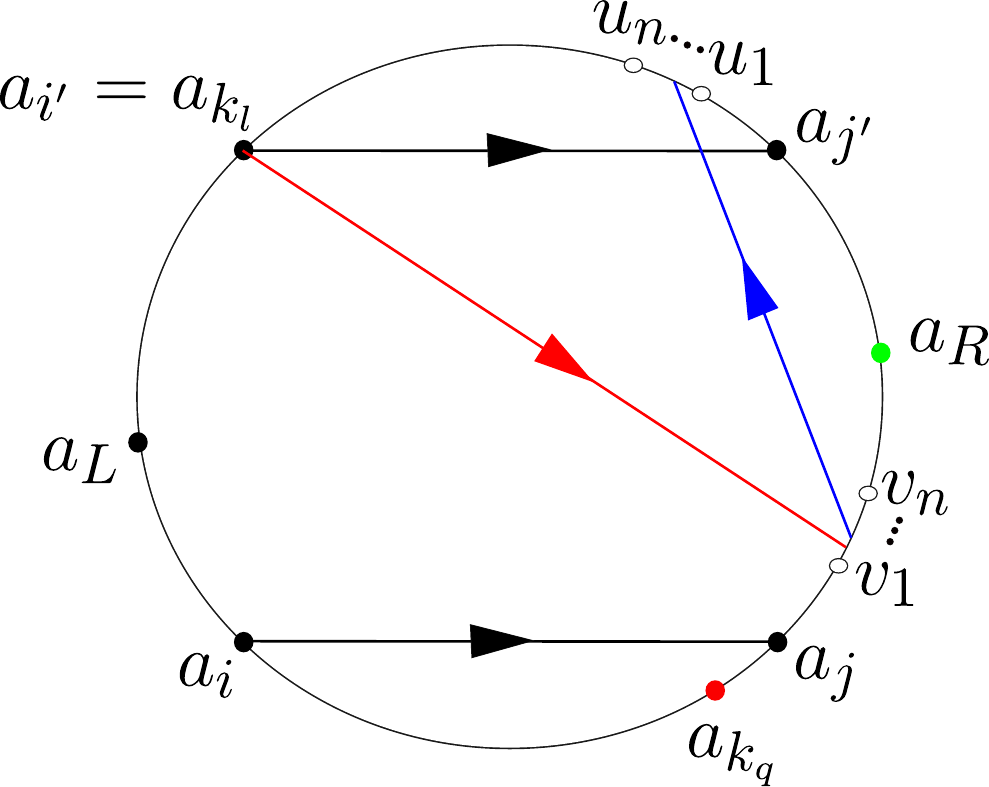}
\caption{$a_L \in A_L$ and $a_R \in A_R$.}
\label{figure:proof}
\end{figure}

\begin{enumerate}
\item
When $c_s =a_{i'}$, recall $\mathbf{I}= \{k_1, \cdots, k_n\}$ with $k_1<\cdots<k_n$ and suppose $i'=k_l$. 
By Lemma \ref{swcal} Equation (\ref{equation:R}), we use $\Delta^R(a_{i'} v_t)$ with respect to the right side of $\overrightarrow{a_{i'} v_t}$ as in Figure \ref{figure:proof}, then
\begin{equation*}
\begin{aligned}
&\left\{a_{i'} v_t, \Delta\left(\left(\widehat{a_{i'}},a_{j'}\right), \left(u^n\right)\right)\right\}
\\&= \sum_{d=l+1}^q 1\cdot a_{k_d} v_t \cdot (-1)^{d-l}\cdot \Delta\left(\left(\widehat{a_{k_d}}, a_{j'}\right), \left(u^n\right)\right),
\end{aligned}
\end{equation*}
where $a_{k_q}$ is on the right side of $\overrightarrow{a_{i'} v_t}$ and $a_{k_{q+1}}$ is on the left side of $\overrightarrow{a_{i'} v_t}$.
Let $\widehat{a_{i'}, a_i}$ be $a_{k_1}, \cdots, a_{k_n}$ with $a_{i'},a_{i}$ removed. The sum over $t$ for $a_{i'}$ equals
\begin{equation}
\begin{aligned}
\label{equation:lt}
&\frac{1}{\det M \cdot B}\sum_{d=l+1}^q\sum_{t=1}^n (-1)^{l+t} M_{l t} \cdot  a_{k_d} v_t \cdot (-1)^{d-l}\cdot \Delta\left(\left(\widehat{a_{k_d}}, a_{j'}\right), \left(u^n\right)\right)
\\&= \frac{1}{\det M \cdot B}\sum_{d=l+1}^q (-1)^{n-1+d}\cdot \Delta\left(\left(\widehat{a_{i'},a_{i}}, a_{k_d}, a_{j}\right), \left(v^n\right)\right) \cdot \Delta\left(\left(\widehat{a_{k_d}}, a_{j'}\right), \left(u^n\right)\right).
\end{aligned}
\end{equation}
If $k_d \neq i$, then $a_{k_d} \in \{\widehat{a_{i'}, a_i}\}$. Thus $\Delta\left(\left(\widehat{a_{i'},a_{i}}, a_{k_d}, a_{j}\right), \left(u^n\right)\right)=0$. Hence the right hand side of Equation~\eqref{equation:lt} equals 

\begin{equation}
\label{equation:exp}
 \frac{\Delta\left(\left(\widehat{a_{i'}},a_{j}\right) , \left(v^n\right)\right)\cdot \Delta\left(\left(\widehat{a_{i}},a_{j'}\right), \left(u^n\right)\right)}{\Delta\left(\left(\widehat{a_i},a_j\right), \left(v^n\right)\right)\cdot \Delta\left(\left(\widehat{a_{i'}},a_{j'}\right), \left(u^n\right)\right)}.
\end{equation}
\item When $c_s=a_j$, we get 
\[\left\{a_j v_t, \Delta\left(\left(\widehat{a_{i'}},a_{j'}\right), \left(u^n\right)\right)\right\}=0.\]
Thus the sum over $t$ for $a_j$ equals $0$.

\item When $c_s=a_m\in \{a_{k_1},\cdots a_{k_n}\}\backslash \{a_{i}, a_{i'}\}$, let $A_L$ ($A_R$ resp.) be the subset of $\{a_{k_1},\cdots a_{k_n}\}\backslash \{a_{i}, a_{i'}\}$ which is on the left (right resp.) side of $\overrightarrow{v_s u_t}$ as in Figure \ref{figure:proof}. When we compute $\left\{a_m v_t, \Delta\left(\left(\widehat{a_{i'}},a_{j'}\right), \left(u^n\right)\right)\right\}$ by Lemma \ref{swcal} Equation (\ref{equation:R}), each term corresponds to $a_m$ swapping with one of the element $a_p$ in $\{\widehat{a_{i'}},a_{j'}\}$. If $a_m\neq a_p$, then $a_m\in\{\widehat{a_{i'}},a_{j'}\}\backslash \{a_p\}$. Thus $\Delta\left(\left(\widehat{a_{i'},a_p},a_m,a_{j'}\right), \left(u^n\right)\right)=0$. Hence we obtain
\begin{equation*}
\left\{a_m v_t, \Delta\left(\left(\widehat{a_{i'}},a_{j'}\right), \left(u^n\right)\right)\right\}=\begin{cases}

   \frac{1}{2}\cdot a_m v_t \cdot \Delta\left(\left(\widehat{a_{i'}},a_{j'}\right), \left(u^n\right)\right)  &\mbox{if $a_m \in A_L$,}\\

-\frac{1}{2}\cdot a_m v_t \cdot \Delta\left(\left(\widehat{a_{i'}},a_{j'}\right), \left(u^n\right)\right)  &\mbox{if $a_m \in A_R$.}
\end{cases} 
\end{equation*}
Thus the sum over $t$ for $a_m$ equals $\frac{1}{2}$ if $a_m\in A_L$, equals $-\frac{1}{2}$ if $a_m\in A_R$.
\end{enumerate}
Suppose $\#A_L=n_1$ and $\#A_R=n_2$. Combining the above results, we obtain
\begin{equation*}
\begin{aligned}
&\left[\Delta\left(\left(\widehat{a_i}|a_j\right), \left(v^n\right)\right), \Delta\left(\left(\widehat{a_{i'}}|a_{j'}\right), \left(u^n\right)\right)\right]
\\&=  \frac{ \Delta\left(\left(\widehat{a_{i'}},a_{j}\right), \left(v^n\right)\right) \cdot \Delta\left(\left(\widehat{a_{i}},a_{j'}\right) , \left(u^n\right)\right)}{\Delta\left(\left(\widehat{a_i},a_j\right), \left(v^n\right)\right)\cdot \Delta\left(\left(\widehat{a_{i'}},a_{j'}\right), \left(u^n\right)\right) }+  \frac{n_1-n_2}{2}
\\&= \frac{ \Delta\left(\left(\widehat{a_{i'}},a_{j}\right), \left(u^n\right)\right) \cdot \Delta\left(\left(\widehat{a_{i}},a_{j'}\right) , \left(u^n\right)\right)}{\Delta\left(\left(\widehat{a_i},a_j\right), \left(u^n\right)\right)\cdot \Delta\left(\left(\widehat{a_{i'}},a_{j'}\right), \left(u^n\right)\right) }+ \frac{n_1-n_2}{2}
\\&= \frac{E\left(\widehat{a_{i'}}|a_{j},a_{i'}\right) \cdot E\left(\widehat{a_{i}}|a_{j'},a_{i}\right)}{E\left(\widehat{a_i}|a_j,a_i\right)\cdot E\left(\widehat{a_{i'}}|a_{j'},a_{i'}\right)} +  \frac{n_1-n_2}{2}.
\end{aligned}
\end{equation*}
By similar computations, we get
\begin{equation*}
\begin{aligned}
&\left[\Delta\left(\left(\widehat{a_i},a_j\right), \left(v^n\right)\right), \Delta\left(\left(\widehat{a_{i'}},a_{i'}\right), \left(u^n\right)\right)\right]
= \frac{n_1+1-n_2}{2},
\end{aligned}
\end{equation*}

\begin{equation*}
\begin{aligned}
&\left[\Delta\left(\left(\widehat{a_i},a_i\right), \left(v^n\right)\right), \Delta\left(\left(\widehat{a_{i'}},a_{j'}\right), \left(u^n\right)\right)\right]
= \frac{n_1+1-n_2}{2},
\end{aligned}
\end{equation*}
and
\begin{equation*}
\begin{aligned}
&\left[\Delta\left(\left(\widehat{a_i},a_i\right), \left(v^n\right)\right), \Delta\left(\left(\widehat{a_{i'}},a_{i'}\right), \left(u^n\right)\right)\right]
= \frac{n_1+2-n_2}{2}.
\end{aligned}
\end{equation*}
Thus
\begin{equation*}
\begin{aligned}
&\left[E\left(\widehat{a_i}|a_j,a_i\right), E\left(\widehat{a_{i'}}|a_{j'},a_{i'}\right)\right]
=  \frac{E\left(\widehat{a_i}|a_{j'},a_i\right) \cdot E\left(\widehat{a_{i'}}|a_{j},a_{i'}\right)}{E\left(\widehat{a_i}|a_j,a_i\right)\cdot E\left(\widehat{a_{i'}}|a_{j'},a_{i'}\right)}+ \\& +  \frac{n_1-n_2}{2} -  \frac{n_1+1-n_2}{2} - \frac{n_1+1-n_2}{2} +   \frac{n_1+2-n_2}{2}
\\& = \frac{E\left(\widehat{a_i}|a_{j'},a_i\right) \cdot E\left(\widehat{a_{i'}}|a_{j},a_{i'}\right)}{E\left(\widehat{a_i}|a_j,a_i\right)\cdot E\left(\widehat{a_{i'}}|a_{j'},a_{i'}\right)}.
\end{aligned}
\end{equation*}
In this case, since $s_{||}\left(a_ia_j,a_{i'}a_{j'}\right)=-1$, we obtain
\begin{equation*}
\begin{aligned}
\left\{E\left(\widehat{a_i}|a_j,a_i\right), E\left(\widehat{a_{i'}}|a_{j'},a_{i'}\right)\right\}
= - s_{||}\left(a_i a_j,a_{i'} a_{j'}\right)\cdot E\left(\widehat{a_i}|a_{j'},a_i\right) \cdot E\left(\widehat{a_{i'}}|a_{j},a_{i'}\right).
\end{aligned}
\end{equation*}
The proof for the other cases are similar with respect to the arrangements in Figure \ref{figure:parallel}. All the terms are cancelled out except the term in Equation~\eqref{equation:exp} is replaced by
\begin{equation*}
 - s_{||}\left(a_i,a_j,a_{i'},a_{j'}\right)\cdot\frac{\Delta\left(\left(\widehat{a_{i'}},a_{j}\right) , \left(v^n\right)\right)\cdot \Delta\left(\left(\widehat{a_{i}},a_{j'}\right), \left(u^n\right)\right)}{\Delta\left(\left(\widehat{a_i},a_j\right), \left(v^n\right)\right)\cdot \Delta\left(\left(\widehat{a_{i'}},a_{j'}\right), \left(u^n\right)\right)}.
\end{equation*}
Note that for the case (3) in Figure \ref{figure:parallel}, the contribution is the same as above, but coming from the fact
\begin{equation*}
\left[a_j v_t, \Delta\left(\left(\widehat{a_{i'}},a_{j'}\right), \left(u^n\right)\right)\right]=-\frac{1}{2}.
\end{equation*}

Hence we obtain for all the cases
\begin{equation*}
\begin{aligned}
\left\{E\left(\widehat{a_i}|a_j,a_i\right), E\left(\widehat{a_{i'}}|a_{j'},a_{i'}\right)\right\}
= - s_{||}\left(a_i,a_j,a_{i'},a_{j'}\right)\cdot E\left(\widehat{a_i}|a_{j'},a_i\right) \cdot E\left(\widehat{a_{i'}}|a_{j},a_{i'}\right).
\end{aligned}
\end{equation*}

Combing with Lemma \ref{lemma:01} and Equation~\eqref{equation:pgnr}, we conclude that $\theta_{\alpha, \beta}$ is Poisson with respect to $\{\cdot,\cdot\}_{B_{\alpha,\beta}}$ and $\{\cdot,\cdot\}_{\beta-\alpha, \alpha + \beta}$. 
\end{proof}

\section*{Acknowledgements}
I thank Yau Mathematical Sciences Center at Tsinghua University and Luxembourg University for their hospitality.

\bibliographystyle{amsplain}

\begin{thebibliography}{99}

\bibitem[BGY06]{BGY06} \author{Brown, K.A., Goodearl and K.R., and Yakimov},
\textit{Poisson structures on affine spaces and flag varieties. I. Matrix affine Poisson space}, Adv. Math., \textbf{206}(2006), 567-629.


\bibitem[CP76]{CP76} \author{C. D. Concini and C. Procesi},
\emph{A Characteristic Free Approach to Invariant Theory}, Advances in Mathematics 21, 330-354(1976).

\bibitem[DS81]{DS81} \author{V. G. Drinfeld and V. V. Sokolov},
\textit{Equations of Korteweg-de Vries type, and simple Lie algebras}, Dokl. Akad. Nauk SSSR \textbf{258} (1981) 11-16.


\bibitem[FZ02]{FZ02}  \author{S. Fomin and A. Zelevinsky},
\textit{Cluster algebras. I: Foundations}, J. Amer. Math. Soc. \textbf{15}(2002), 497-529.

\bibitem[G84]{G84} \author{W. M. Goldman},
\textit{The Symplectic Nature of Fundamental Groups of Surfaces}, Adv. Math. \textbf{54} (1984), 200-225.


\bibitem[GSV03]{GSV03} \author{M. Gekhtman, M. Shapiro and A. Vainshtein},
\textit{Cluster algebras and Poisson geometry}, Mosc. Math. J., \textbf{3}(2003), 899-934.

\bibitem[GSV09]{GSV09} \author{M. Gekhtman, M. Shapiro and A. Vainshtein},
\textit{Poisson Geometry of Directed Networks in a Disk}, Selecta Math. \textbf{15} (2009), 61-103.

\bibitem[GSSV12]{GSSV12} \author{M. Gekhtman, M. Shapiro, A. Stolin and A. Vainshtein},
\textit{Poisson Structures Compatible with the Cluster Algebra Structure in Grassmannians
}, Letters in Mathematical Physics 100.2 (2012): 139-150.

\bibitem[L18]{L18} \author{F. Labourie},
\textit{Goldman algebra, opers and the swapping algebra}, Geometry \& Topology 22.3 (2018): 1267-1348.

\bibitem[Po06]{Po06} \author{A. Postnikov},
\textit{Total positivity, Grassmannians and networks}, preprint(2006), arXiv:0609764.

\bibitem[Se83]{Se83} \author{M. A. Semenov-Tyan-Shanskii},
\textit{What is a classical r-matrix?},
Funct. Anal. Appl., \textbf{17}(1983), 259-272.


\bibitem[Su17]{Su17} \author{Z. Sun},
\textit{Rank n swapping algebra for the $\operatorname{PSL}(n,\mathbb{R})$ Hitchin component}, Int Math Res Notices 2017 (2): 583-613.

\bibitem[Su15]{Su15} \author{Z. Sun},
\textit{Rank n swapping algebra for $\operatorname{PGL}_n$ Fock--Goncharov $\mathcal{X}$ moduli space}, preprint(2015), arXiv:1503.00918.


\bibitem[W39]{W39} \author{H. Weyl},
\emph{The Classical Groups: Their Invariants and Representations} , Princeton university press,  (1939), 320 pages.

\end{thebibliography}

\end{document}